\documentclass{elsarticle} 
 
 \usepackage{amsmath,amssymb,amsthm}
\usepackage{verbatim}
\usepackage{tikz}
\usetikzlibrary{matrix,arrows}
\usepackage{url}
\usepackage{hyperref}

\theoremstyle{plain}
\newtheorem{theorem}{Theorem}
\newtheorem{proposition}[theorem]{Proposition}
\newtheorem{lemma}[theorem]{Lemma}
\newtheorem{corollary}[theorem]{Corollary}

\theoremstyle{definition}
\newtheorem{definition}[theorem]{Definition}

\newtheorem{remark}[theorem]{Remark}

\numberwithin{theorem}{section}

\newcommand{\into}{\hookrightarrow}

\renewcommand{\theta}{\vartheta}
\renewcommand{\phi}{\varphi}

\newcommand\op{\mathrm{op}}
\newcommand{\dom}{\mathrm{dom}}
\newcommand{\sB}{\mathcal{B}}
\newcommand{\sS}{\mathcal{S}}
\newcommand{\sT}{\mathcal{T}}
\newcommand{\sF}{\mathcal{F}}
\newcommand{\sV}{\mathcal{V}}
\DeclareMathOperator{\Con}{\mathrm{Con}}
\DeclareMathOperator{\KS}{\mathcal{K}}
\DeclareMathOperator{\im}{\mathrm{im}}
\DeclareMathOperator{\Cl}{\mathrm{Cl}}

\renewcommand{\hat}{\widehat}

\makeatletter
\newcommand\twoheaduparrow{%
\mathrel{\mathchoice
  {\raise2pt\hbox{%
  \ooalign{\hss$\uparrow$\hss\cr\lower2pt\hbox{%
  $\uparrow$}}}}
  {\raise2pt\hbox{%
  \ooalign{\hss$\uparrow$\hss\cr\lower2pt\hbox{%
  $\uparrow$}}}}
  {\raise1.5pt\hbox{%
  \ooalign{\hss$\scriptstyle\uparrow$\hss\cr\lower1.5pt\hbox{%
  $\scriptstyle\uparrow$}}}}
  {\raise1.1pt\hbox{%
  \ooalign{\hss$\scriptscriptstyle\uparrow$\hss\cr\lower1.1pt\hbox{%
  $\scriptscriptstyle\uparrow$}}}}
}}
\newcommand{\wayuparrow}[1]{\twoheaduparrow\!\!{#1}}
    
\newif\ifshowproofs
\showproofsfalse

\begin{document}

\title{Sheaves and duality}

\author{Mai Gehrke\corref{cor}\fnref{ERC-Adv}}
\ead{mgehrke@irif.fr}
\address{IRIF, Universit\'e Paris Diderot -- Paris 7, Case 7014, 
75205 Paris Cedex 13, France}
\cortext[cor]{Corresponding author.}

\author{Samuel J. v. Gool \fnref{ERC-MSC}}
\ead{samvangool@me.com}
\address{Mathematics Department, City College of New York, NY 10031, USA\\
and\\
ILLC, Universiteit van Amsterdam, Postbus 94242, 1090 GE Amsterdam, The Netherlands}

\fntext[ERC-Adv]{The research of this author has received funding from the European Research Council (ERC) under the European Union's Horizon 2020 research and innovation programme under the ERC Advanced grant agreement No.670624.}
\fntext[ERC-MSC]{The research of this author has received funding from the European Research Council (ERC) under the European Union's Horizon 2020 research and innovation programme under the Marie Sklodowska-Curie grant agreement No.655941.}

\date{\today}



\begin{abstract}
It has long been known in universal algebra that any distributive sublattice of congruences of an algebra which consists entirely of commuting congruences yields a sheaf representation of the algebra. In this paper we provide a generalisation of this fact and prove a converse of the generalisation. To be precise, we exhibit a one-to-one correspondence (up to isomorphism) between soft sheaf representations of universal algebras over stably compact spaces and frame homomorphisms from the dual frames of such spaces into subframes of pairwise commuting congruences of the congruence lattices of the universal algebras. For distributive-lattice-ordered algebras this allows us to dualize such sheaf representations. 
\end{abstract}

\maketitle

\section{Introduction}\label{sec:intro}
Sheaf theory emerged in the 1950's and is still central to cohomology theory. Sheaves, as generalized Stone duality, have also found applications in logic and model theory. 
Since the 1970's, sheaf representation results for universal algebras, and in particular for lattice-ordered algebras, have been studied at various levels of generality, and the existence of a distributive lattice of pairwise commuting congruences has previously been identified as an essential ingredient for a good sheaf representation. 

Our intention in this paper is to 
identify \emph{exactly which sheaf representations} correspond to a distributive lattice of pairwise commuting congruences. Our main contribution is to identify the notion of \emph{softness}, which originated with Godement's treatment of homological algebra \cite{God1958}, as central in this respect. 

Our work here grew out of our work on sheaf representations of MV-algebras with Marra \cite{GGM2014} and as such it is closely related to recent work on sheaf representations for MV-algebras \cite{DP2010,RuYa2008} and $\ell$-groups \cite{Schwartz2013}. Softness has also been considered in the study of Gelfand rings, see \cite{Bko70,Mulvey,BanVer2011}, and more recently \cite{SchTre2010}.

An important feature that is essential for applications is that we allow the base spaces of the sheaves we consider to be non-Hausdorff. On the other hand, a tight relationship between the open and the compact sets of the base spaces is required for our results. A natural class of spaces, whose features are particularly well adapted to our results, are the \emph{stably compact spaces} \cite{Lawson2011,GHK03}. This class of topological spaces, which is closely related to Nachbin's compact ordered spaces, provides a common generalization of compact Hausdorff spaces and spectral spaces. Stably compact topologies naturally come with an associated dual topology; the two topologies are related by being the open up-sets and the open down-sets, respectively, of the topology of a compact ordered space. 

This so-called \emph{co-compact duality} for stably compact spaces plays a prominent role in our main result (Theorem~\ref{thm:main}):
\emph{soft sheaf representations of an algebra over a stably compact base space correspond bijectively to frame homomorphisms from the open set lattice of the co-compact dual of the frame of the base space to a frame of commuting congruences of the algebra}. 
Congruence lattices are not frames in general. By frame homomorphisms into the congruence lattice we mean a map preserving finite meets and arbitrary joins. Since open set lattices are frames it will follow that the image of such a map is a frame in the inherited operations.

Our main application of this result is that soft sheaf representations of a distributive lattice correspond bijectively to continuous decompositions of its Priestley dual space which satisfy an `interpolation' property that we introduce (Theorem~\ref{thm:DLmain}). Applying this result, we also obtain 
a general framework for previously known results on sheaf representations of MV-algebras.

The paper is organized as follows. In Section~\ref{sec:scs} we give the necessary background on stably compact spaces. In Section~\ref{sec:main} we prove our main theorem on soft sheaves as morphisms into the congruence lattices of universal algebras. In Section~\ref{sec:morphisms} we show how direct image sheaves fare under our correspondence. In Section~\ref{sec:applications} we apply our result to distributive-lattice-ordered algebras.

\section{Stably compact spaces}\label{sec:scs}
We identify here the main technical facts about topological and ordered topological spaces that we will need.

A \emph{compact ordered space} is a tuple $(Y,\pi,\leq)$ where $(Y,\pi)$ is a compact Hausdorff space and $\leq$ is a partial order on $Y$ which is a closed subset of the product space $(Y,\pi) \times (Y,\pi)$. Given a compact ordered space $(Y,\pi,\leq)$, we denote by $\pi^{\uparrow}$ the topology on $Y$ consisting of $\pi$-open \emph{up-sets}, i.e., $\pi$-open sets which are moreover upward closed in the partial order $\leq$, and by $\pi^{\downarrow}$ the topology on $Y$ consisting of $\pi$-open down-sets. Given a compact ordered space $(Y,\pi,\leq)$, we write $Y^\uparrow$ for the topological space $(Y,\pi^{\uparrow})$, and $Y^{\downarrow}$ for the space $(Y,\pi^{\downarrow})$. Both $Y^{\uparrow}$ and $Y^{\downarrow}$ are so-called \emph{stably compact spaces}, for which see, e.g., \cite[Sec.~VI-6]{GHK03}, \cite{Lawson2011} and the references therein.\footnote{Note that if $(Y,\pi,\leq)$ is a compact ordered space, then so is $Y^{\op}=(Y,\pi,\geq)$ and thus $Y^{\downarrow}$ is $(Y^{\op})^\uparrow$ so that these constructions give rise to the same class of spaces.} In fact, every stably compact space arises as $Y^\uparrow$ for a unique compact ordered space with the same underlying set; cf., e.g., \cite[Prop.~2.10]{Lawson2011}. The order on this compact ordered space $Y$ is the \emph{specialization order} of $Y^{\uparrow}$, defined by $x \leq y$ iff every open set containing $x$ also contains $y$. The topology of a stably compact space can be characterized as a stably continuous frame with compact top element \cite[Sec.~VI-7]{GHK03}.

A subset of a topological space is said to be \emph{saturated} provided it is an intersection of open sets. In a $T_1$ space, and thus in particular in a Hausdorff space, every subset is saturated. In general, the saturated subsets of a topological space are the up-sets of its specialization order.
The notion of saturated sets is central to toggling between the topological spaces $Y^\uparrow$ and $Y^{\downarrow}$:
\begin{proposition}\label{prop:toggle}
Let $(Y,\pi,\leq)$ be a compact ordered space. For any subset $S \subseteq Y$, the following are equivalent:
\begin{enumerate}
\item $S$ is a closed up-set in $(Y,\pi,\leq)$,
\item $S$ is compact and saturated in $(Y,\pi^{\uparrow})$,
\item $S$ is closed in $(Y,\pi^{\downarrow})$.
\end{enumerate}
In particular, the complements of compact-saturated sets of $\pi^{\uparrow}$ are exactly the open sets of $\pi^{\downarrow}$.
\end{proposition}
\begin{proof}
See, e.g., \cite[Lem.~2.4 \& Thm.~2.12]{Jung2004}.
\end{proof}
Stably compact spaces can be characterized intrinsically: they are those topological spaces which are $T_0$, compact, locally compact, \emph{coherent} (the intersection of compact-saturated sets is compact) and \emph{sober} (the only  union-irreducible closed sets are closures of points), see, e.g.,~\cite[Subsec.~2.3]{Jung2004}.
The fact that stably compact spaces are in particular sober will allow us to apply the celebrated \emph{Hofmann-Mislove Theorem}, which we will recall now.

Let $\Omega$ be a frame (that is, a complete lattice in which binary meets distribute over arbitrary joins). A filter $\sF \subseteq \Omega$ is called \emph{Scott-open} if, for any directed $(u_i)_{i \in I}$ in $\Omega$ such that $\bigvee_{i \in I} u_i \in \sF$, there exists $i \in I$ such that $u_i \in \sF$. We denote by $\mathrm{Filt}(\Omega)$ the lattice of filters of $\Omega$, ordered by inclusion, and by $\sigma\mathrm{Filt}(\Omega)$ the lattice of Scott-open filters of $\Omega$ ordered by inclusion. 

If $Y$ is a topological space, we denote by $\Omega Y$ the set of opens of $Y$, and by $\KS Y$ the set of compact-saturated subsets of $Y$, and both are partially ordered by inclusion.

\begin{theorem}[Hofmann-Mislove Theorem]\label{thm:hofmis}
Let $Y$ be a sober space. The function $\phi\colon\KS Y \to \mathrm{Filt}(\Omega Y)$ defined for $K \in \KS Y$ by 
\[ K \mapsto \sF_K := \{U \in \Omega Y \ | \ K \subseteq U\} \in \mathrm{Filt}(\Omega Y)\]
is an order-embedding whose image consists precisely of the Scott-open filters. In particular, given a Scott-open filter $\sF$ of $\Omega Y^{\uparrow}$, the intersection, $K_\sF := \bigcap\sF$, is the unique compact-saturated set in $Y^{\uparrow}$ such that $\phi(K_\sF) = \sF$.
\end{theorem}
\begin{proof}
This is \cite[Theorem 2.16]{HM1981}. Also see \cite{KP1994} for a shorter proof.
\end{proof}

Let $Y$ be a locally compact space. Recall, see e.g. \cite[Prop.~3.3]{Lawson2011}, that, for $U, U' \in \Omega Y$, we have that $U$ is \emph{way below} $U'$, denoted $U \prec U'$, if, and only if, there exists $K \in \KS Y$ such that $U \subseteq K \subseteq U'$. For $K, K' \in \KS Y$, we will also write $K \prec K'$ if there exists $U \in \Omega Y$ such that $K \subseteq U \subseteq K'$.\footnote{Note that this notation is consistent with the use of the same symbol `$\prec$' for the way below relation between opens: if either $S$ or $S'$ is compact and open, and the other is compact or open, then $S \prec S'$ if, and only if, $S \subseteq S'$, for either interpretation of the symbol `$\prec$'.}

We recall three topological facts that we need in the proof of our main theorem in the next section. We include the short proofs for the sake of completeness. The first of these facts is a property of locally compact spaces that has been called \emph{Wilker's condition} in the literature \cite{KeiLaw05}. 

\begin{lemma} \label{lem:refinecover}
Let $Y$ be a locally compact space. If $K$ is compact-saturated and $(V_i)_{i = 1}^n$ is a finite open cover of $K$, then there exists a finite open cover $(U_i)_{i=1}^n$ of $K$ such that $U_i \prec V_i$ for each $i = 1, \dots, n$.
\end{lemma}
\begin{proof}
For each $y \in K$, pick some $i(y)$ such that $y \in V_{i(y)}$, and by local compactness of $Y$ pick an open $U_y$ such that $y \in U_y \prec V_{i(y)}$. Then $(U_y)_{y \in K}$ is an open cover of $K$, so pick $S \subseteq K$ finite such that $(U_y)_{y \in S}$ covers $K$. For each $i = 1,\dots,n$, define $U_i := \bigcup \{ U_y \ | \ y \in S, i(y) = i\}$. Then $(U_i)_{i = 1}^n$ is an open cover of $K$ and $U_i \prec V_i$ for each $i$.
\end{proof}

The second property we will need is specific to stably compact spaces, and is called \emph{weakly Hausdorff} in the literature \cite{KeiLaw05}. 

\begin{lemma}\label{lem:dualrefinecover}
Let $(Y,\leq,\tau)$ be a compact ordered space. Let $K_1, \dots, K_n \in \KS Y^{\uparrow}$ and $U \in \Omega Y^{\uparrow}$ such that $\bigcap_{i=1}^n K_i \subseteq U$. There exist $L_1, \dots, L_n \in \KS Y^{\uparrow}$ such that $K_i \prec L_i$ and $\bigcap_{i=1}^n L_i \subseteq U$.
\end{lemma}
\begin{proof}
By Proposition~\ref{prop:toggle}, $(Y \setminus K_i)_{i=1}^n$ is an open cover in $Y^\downarrow$ of the set $Y \setminus U$ which is compact-saturated in $Y^\downarrow$. Apply Lemma~\ref{lem:refinecover} to $Y^\downarrow$ to obtain a finite $Y^{\downarrow}$-open cover $(V_i)_{i=1}^n$ of $Y \setminus U$ such that $V_i \prec Y \setminus K_i$. For each $i$, pick $M_i \in \KS Y^{\downarrow}$ such that $V_i \subseteq M_i \subseteq Y \setminus K_i$. Defining $L_i := Y \setminus V_i$ now gives the result.
\end{proof}

Finally, the third property we need is closely related to the frame-theoretic characterization of stably compact spaces as stably continuous frames with compact top element \cite[Section~VI-7]{GHK03}.

\begin{lemma}\label{lem:compsatasintersection}
Let $(Y,\pi,\leq)$ be a compact ordered space. For any compact-saturated set $K$ in $Y^{\uparrow}$, the collection $\wayuparrow{K} := \{ K' \in \KS Y^{\uparrow} \ | \ K \prec K'\}$ is filtered, and $K = \bigcap \wayuparrow{K}$.
\end{lemma}
\begin{proof}
If $K_1, K_2 \in \KS Y^{\uparrow}$ are such that $K \prec K_i$, pick $U_1, U_2 \in \Omega Y^{\uparrow}$ such that $K \subseteq U_i \subseteq K_i$. Then, as $Y^{\uparrow}$ is coherent, $K_1 \cap K_2\in\KS Y^{\uparrow}$ and $U_1 \cap U_2$ witnesses that $K \prec K_1 \cap K_2$, so $\wayuparrow{K}$ is filtered. Clearly $K \subseteq \bigcap \wayuparrow{K}$. For the reverse inclusion, suppose that $y \not\in K$. As $K$ is saturated, it is an intersection of open sets, so there is $V\in\Omega Y^{\uparrow}$ with $K\subseteq V$ and $y\not\in V$. By Lemma~\ref{lem:refinecover} with $n=1$, there is $U\in\Omega Y^{\uparrow}$ with $K\subseteq U\prec V$. It follows that there is $K'\in\KS Y^{\uparrow}$ with $U\subseteq K'\subseteq V$. Thus $K' {\in} \wayuparrow{K}$ and $y\not\in K$ so that $K = \bigcap \wayuparrow{K}$.
\end{proof}

\section{Sheaves and congruences}\label{sec:main}
We are interested in sheaf representations of algebras, and for the work on sheaves we need an ambient category, in which we will assume that products and subobjects are given by Cartesian products and (isomorphic copies of) subalgebras. We will also need colimits, and thus it is natural to assume we are working in a category $\sV$ which is a variety of algebras of some 
finitary 
signature with their algebra homomorphisms.

A \emph{presheaf of $\sV$-algebras} over a topological space $Y$ is a functor $F : (\Omega Y)^\op \to \sV$. Given a collection $(U_i)_{i \in I}$ of opens of $Y$ and a  collection  $(s_i)_{i \in I}$ with $s_i \in FU_i$ for each $i \in I$, we say that the $s_i$ are \emph{patching} provided for any $i,j\in I$ we have\footnote{As usual in sheaf theory, if $U' \subseteq U$ are open sets in $Y$, we use the notation $s|_{U'}$ for the image of an element $s \in FU$ under the map obtained by applying $F$ to the inclusion $U' \subseteq U$.} 
\[
s_i|_{U_i\cap U_j}=s_j|_{U_i\cap U_j}.%
\]
A presheaf is a \emph{sheaf} provided it satisfies the \emph{patch property}: any patching family extends uniquely to the union of their domains. That is, for any collection of opens $(U_i)_{i \in I}$ of $Y$ and $(s_i)_{i \in I}$ with $s_i \in FU_i$ for each $i \in I$, so that the $s_i$ are patching, there exists a unique $s \in F(\bigcup_{i \in I} U_i)$ such that $s|_{U_i} = s_i$ for all $i \in I$.%

A closely related notion is that of a bundle of $\sV$-algebras. A \emph{bundle of $\sV$-algebras} over a space $Y$ is a continuous map $p \colon E \to Y$ together with, for each $y \in Y$ and each $n$-ary operation symbol $f$ of $\sV$-algebras, an operation $f^{E_y} : (E_y)^n \to E_y$, where $E_y := p^{-1}(y)$, in such a way that $(E_y, (f^{E_y}))$ is a $\sV$-algebra, and such that the partial operation $f^E$ from $E^n$ to $E$, defined as the union of the functions $f^{E_y}$, is continuous. For each $y \in Y$, the topological $\sV$-algebra $E_y$ is called the \emph{stalk} at of the bundle at $y$. Given  an open $U\subseteq Y$, a continuous function $s : U \to E$ such that $p s = \mathrm{id}_U$ is called a \emph{local section} of $p$ over $U$. A \emph{global section} is a local section whose domain is $Y$.

Given a bundle $E \to Y$ of $\sV$-algebras, assign to every open set $U$ of the base space $Y$ the $\sV$-algebra $FU$ of local sections over $U$, the subalgebra of the direct product $\prod_{y \in U} E_y$ consisting of the continuous functions. In the case $U = Y$, the algebra $FY$ is called the \emph{algebra of global sections} of $F$. This assignment on objects extends to a sheaf of $\sV$-algebras by letting $F(U \subseteq V)$ send a local section over $V$ to its restriction over $U$.  There is a reverse process which assigns to every presheaf $F$ of $\sV$-algebras a bundle $p \colon E \to Y$ of $\sV$-algebras. This bundle will always be a so-called \emph{\'etale space}, that is, a bundle $p$ for which every $e \in E$ has an open neighborhood $V$ such that $pV$ is open and $p|_V \colon V \to pV$ is a homeomorphism, see \cite[Chapter\,II.5]{MM} for the construction of the \'etale space associated with a sheaf. These two assignments extend to an adjunction between bundles and presheaves, which restricts to an equivalence of categories between sheaves of $\sV$-algebras and \'etale spaces of $\sV$-algebras \cite[Thm.~II.6.2]{MM}. In the sequel it will be useful to know that, in an \'etale space, there is a local section through any point $e \in E$, and both the function $p$ and any local section $s$ of $p$ are open mappings \cite[Prop.~II.6.1]{MM}.

\begin{definition}
A \emph{sheaf representation} of a $\mathcal{V}$-algebra $A$ is a sheaf $F$ such that $A$ is isomorphic to $FY$, the algebra of global sections of $F$.
\end{definition}

In this case, $A$ embeds into the direct product $\prod_{y \in Y} E_y$, where $E_y$ is the stalk at $y \in Y$ of the \'etal\'e space corresponding to $F$.

For $U\subseteq Y$ open and $s,t\in FU$, we write $\|s = t\|$ for the \emph{equalizer} of $s$ and $t$, which is defined as the set of those $y \in U$ so that there exists an open $V$ with $y\in V\subseteq U$ and $s|_V=t|_V$. It is clear from the definition that equalizers are always open. It is also not hard to see that in the setting of the \'etale space $p \colon E \to Y$ corresponding to $F$, the equalizer of two local sections consists of the set of points $y\in U$ at which they take the same value.

In the \'etale space formulation of sheaves, one can consider continuous sections over subsets of $Y$ which are not necessarily open. 
Given a sheaf representation $F$ of an algebra $A$ with corresponding \'etale space $p\colon E\to Y$, we define for each subset $S\subseteq Y$ the subalgebra 
\[
\Gamma S := \{s\colon S\to E\mid s \text{ is a continuous section of  } p\}
\]
 of the direct product $\prod_{y \in S} E_y$. For each open $U\subseteq Y$ we have $\Gamma U\cong FU$. 
Also, for $S\subseteq T\subseteq Y$, we denote the restriction morphism $\Gamma T \to\Gamma S $ by $h_{TS}$.

We now introduce the notion of `softness' \cite[Sec.~II.3.4]{God1958} which is appropriate in our context.
\begin{definition}\label{def:soft}
Let $F$ be a sheaf of $\sV$-algebras over a space $Y$ and let $p \colon E \to Y$ be the corresponding \'etale space. Then $F$ is called \emph{soft} if, for every compact saturated $K \subseteq Y$ and continuous section $s \colon K \to E$ of $p$, there exists a global section $t$ of $p$ such that $t|_K = s$.
\end{definition}
\begin{remark} \ \\
1. In the special case where the base space is assumed to be locally compact and Hausdorff, Definition~\ref{def:soft} remains the same if `compact saturated' is replaced by `closed', cf.~\cite[Prop.~2.5.6]{KS1994}. However, our underlying space may fail to be Hausdorff, and the definition with compact sets, rather than closed sets, turns out to be the appropriate one for our purposes. What we call `soft' here is sometimes called `c-soft', but since we never use the competing notion in this paper, no confusion will arise.

2. In Definition~\ref{def:soft} we define `soft' using the \'etale space of a sheaf. An equivalent definition which directly uses the functor is: $F$ is \emph{soft} if, and only if, for every Scott-open filter $\sF$ in $\Omega Y$, $U \in \sF$ and $s \in FU$, there exists $t \in FY$ such that, for some $V \in \sF$ with $V \subseteq U$, $s|_{V} = t|_{V}$. We leave it as an exercise for the interested reader to prove that this definition is indeed equivalent.
\end{remark}

Before we can get to our main results, we need the following lemma, which shows how to recover the value of a sheaf on open sets given its value on compact-saturated sets.
\begin{lemma}\label{lem:sheafiso}
Let $Y$ be a locally compact space, and $F$ a sheaf of $\sV$-algebras over $Y$. For each open $U$ in $Y$, $FU$ is the inverse limit of the filtering diagram of maps $h_{KL}\colon\Gamma K \to \Gamma L$, where $K, L \in \KS Y, L \subseteq K \subseteq U$. For any $U \subseteq V$ open in $Y$, the restriction map $FV \to FU$ is given by the universal property of the inverse limit $FU$. 
\end{lemma}
\begin{proof}
For each $K\subseteq U$, we have the restriction map $h_{UK}\colon FU \to\Gamma K$ and these commute with the restriction maps $h_{KL}\colon\Gamma K \to \Gamma L$. We prove that $(h_{UK} \colon FU \to \Gamma K)_{K \subseteq U}$ is the inverse limit.

Let $(s_K)_{K \subseteq U}$ be a consistent family for the diagram for $U$. By Lemma~\ref{lem:refinecover} in the case $n = 1$, for each $K \subseteq U$, pick $V_K$ open and $M_K$ compact-saturated in $Y$ such that $K \subseteq V_K \subseteq M_K \subseteq U$. Note that $(V_K)_{K \subseteq U}$ is an open covering of $U$ and, since $(s_K)_{K \subseteq U}$ is a consistent family, $(s_{M_K}|_{V_K})_{K \subseteq U}$ is a patching family of local sections. By the patching property of $F$, there is a unique section $s \in FU$ such that $s|_{V_K} = s_{M_K}|_{V_K}$ for each $K \subseteq U$. Since $(s_K)_{K \subseteq U}$ is a consistent family, it follows that $s|_K = s_K$ for each $K \subseteq U$, and $s$ is clearly the unique such section.
Finally suppose $V\subseteq U$ are opens of $Y$. The restriction map $(-)|_V\colon FU\to FV$ is carried by the isomorphisms $FU \cong \Gamma U$ and $FV \cong \Gamma V$ to the restriction map $h_{UV}\colon\Gamma U \to\Gamma V$, which is given uniquely by the universal property of the limit $\Gamma V$ since $\{h_{KL} \ | \ L \subseteq K \subseteq V\}\subseteq\{h_{KL} \ | \ L \subseteq K \subseteq U\}$.
\end{proof}

In the following proposition, we associate to a sheaf representation $F$ of an algebra $A$ a function $\theta_F$ into the congruence lattice of $A$. 
Although $\Con A$ is not in general a frame, we will call a function into $\Con A$ a \emph{frame homomorphism} if, and only if, it preserves finite meets and arbitrary joins. Note that it follows that the image of such a function will be a ($\wedge$,$\bigvee$)-substructure of $\Con A$ and a frame.
\begin{proposition}\label{prop:deftheta}
Let $A$ be an algebra and let $Y$ be a compact ordered space. For any soft sheaf representation $F$ of $A$ over $Y^{\uparrow}$ and any $K\in \KS Y^{\uparrow}$, the set 
\[ \theta_F(K) := \{(a,b) \in A \times A \ | \ K \subseteq \|a = b\|\}\]
is a congruence of $A$ and the ensuing map 
\[
\theta_F \colon (\KS Y^{\uparrow})^\op \to \Con A, \ K\mapsto \theta_F(K)
\]
is a frame homomorphism %
for which any two congruences in the image commute. 
\end{proposition}
\begin{proof}
Let $F$ be a soft sheaf representation of $A$. We identify $A$ with its image under the isomorphism between $A$ and $FY$. Denote by $e \colon A \times A \to \Omega Y^{\uparrow}$ the function which assigns to $(a,b) \in A \times A$ the open set $\|a = b\|$. 
Notice that, for any $K \in \KS Y^{\uparrow}$, 
\begin{equation}\label{eq:thetasof}
\theta_F(K) = e^{-1}(\sF_K),
\end{equation} 
where $\sF_K$ denotes the Scott-open filter $\{U \in \Omega \ | \ K \subseteq U\}$ corresponding to $K$ (Theorem~\ref{thm:hofmis}).
It is straight-forward to check that, for any filter $\sF$ in $\Omega Y^{\uparrow}$, $e^{-1}(\sF)$ is a congruence, so $\theta_F$ is well-defined. Since $e^{-1}$, viewed as a map from $\mathcal{P}(\Omega Y^{\uparrow})$ to $\mathcal{P}(A \times A)$, preserves arbitrary unions and intersections, and since finite meets and directed joins in both $\sigma\mathrm{Filt}(\Omega Y^{\uparrow})$ and $\mathrm{Con}(A)$  are calculated as finite intersections and directed unions, respectively, it is immediate from (\ref{eq:thetasof}) that $\theta_F$ preserves finite meets and directed joins. 

We now show that, for any $K_1, K_2 \in \KS Y^{\uparrow}$, we have
\begin{equation}\label{eq:joincomp}
\theta_F(K_1 \cap K_2) \subseteq \theta_F(K_1) \circ \theta_F(K_2).
\end{equation}
Suppose that $(a_1,a_2) \in \theta_F(K_1 \cap K_2)$, i.e., $K_1 \cap K_2 \subseteq \|a_1 = a_2\|$. By Lemma~\ref{lem:dualrefinecover}, pick $U_1$, $U_2$ open in $Y^{\uparrow}$ such that $K_i \subseteq U_i$ ($i = 1,2$) and $U_1 \cap U_2 \subseteq \|a_1 = a_2\|$. It follows that $\{a_i|_{U_i}\}_{i=1,2}$ is a compatible family of sections for the covering $\{U_1,U_2\}$ of $U_1 \cup U_2$, so, since $F$ is a sheaf, pick $b \in F(U_1 \cup U_2)$ such that $b|_{U_i} = a_i$ ($i=1,2$). Now $b|_{K_1 \cup K_2}$ is a section over a compact-saturated set, so by softness of $F$, pick a global section $c \in A$ such that $c|_{K_1 \cup K_2} = b|_{K_1 \cup K_2}$. Notice that, for $i = 1, 2$, $a_i|_{K_i} = b|_{K_i} = c|_{K_i}$, so $(a_i,c) \in \theta_F(K_i)$. Thus, $c$ witnesses that $(a_1,a_2) \in \theta_F(K_1) \circ \theta_F(K_2)$, as required.

Combining (\ref{eq:joincomp}) with the inclusions
\[\theta_F(K_1) \circ \theta_F(K_2) \subseteq \theta_F(K_1) \vee \theta_F(K_2) \subseteq \theta_F(K_1 \cap K_2),\]
where the last inclusion holds because $\theta_F$ is order-reversing, we conclude that 
\[\theta_F(K_1 \cap K_2) = \theta_F(K_1) \circ \theta_F(K_2) = \theta_F(K_1) \vee \theta_F(K_2).\] Thus, $\theta_F$ preserves finite joins and any two congruences in the images of $\theta_F$ commute.
\end{proof}
The crucial technical step that we need for our main theorem is to recover a sheaf representation $F$ from the map $\theta_F$ defined in Proposition~\ref{prop:deftheta}. To this end, we make the following definitions. 
\begin{definition}[Sheaf associated to a homomorphism] \label{def:Ftheta}
Let $\theta \colon (\KS Y^{\uparrow})^\op \to \Con(A)$ be a frame homomorphism such that any two congruences in the image of $\theta$ commute. For each $y \in Y$, denote by $\theta_y$ the congruence $\theta({\uparrow}y)$. Define the disjoint union of $\mathcal{V}$-algebras
\[ E_\theta := \bigsqcup_{y \in Y} A/{\theta_y},\]
and let $p \colon E_\theta \to Y$ be the function which maps each summand $A/{\theta_y}$ to its index $y$. For each $a \in A$, denote by $s_a \colon Y \to E_\theta$ the function defined by $s_a(y) := a/{\theta_y}$, for $y \in Y$. Equip $E_\theta$ with the topology generated by the collection
\[\sB=\{s_a(U) \ | \ a \in A, U \in \Omega Y^{\uparrow}\}.\] 
Note that $p \colon E_\theta \to Y^{\uparrow}$ is a continuous function, since for any $U \in \Omega Y^{\uparrow}$, we have $p^{-1}(U) = \bigcup_{a \in A} s_a(U)$. Denote by $F_\theta$ the sheaf of local sections of $p$. 
\end{definition}
It is almost immediate that, for each $y \in Y$, the kernel of the evaluation map $a \mapsto s_a(y)$ is exactly the congruence $\theta({\uparrow} y)$. We now prove that this connection between $\theta$ and $F_{\theta}$ `lifts' from stalks to all compact saturated sets.
\begin{lemma}\label{lem:KStheta}
Let $K\in{\KS} Y^{\uparrow}$ and $a,b\in A$. Then:
\begin{enumerate} 
\item $\theta(K) = \bigcup \{\theta(K') \ | \ K \prec K'\}$.
\item $K\subseteq \|s_a = s_b\|$ if, and only if, $(a,b)\in \theta(K)$.
\end{enumerate}
\end{lemma}
\begin{proof}
1. By Lemma~\ref{lem:compsatasintersection}, since $\theta$ is a frame homomorphism, and directed joins in $\Con A$ are calculated as unions.

2. ($\Rightarrow$) Suppose that $K\subseteq \|s_a = s_b\|$. Then, for each $y\in K$, we have $(a,b) \in \theta_y = \theta({\uparrow}y)$. By the first item, applied to $K = {\uparrow}y$, there exist $U_y  \in \Omega Y^{\uparrow}$ and $K_y\in \KS Y^{\uparrow}$ with $y\in U_y\subseteq K_y$ and $(a,b)\in\theta(K_y)$. Since $K$ is compact, there is a finite subset $M\subseteq K$ so that $K\subseteq \bigcup_{y \in M} U_y \subseteq \bigcup_{y \in M} K_y$.
Thus $(a,b)\in \bigcap \{ \theta(K_y)\mid y\in M\}=\theta\left({\bigcup \{ K_y\mid y\in M\}}\right)\subseteq\theta(K)$, where we have used that $\theta$ preserves finite unions.

($\Leftarrow$) Suppose that $(a,b) \in \theta(K)$. Let $y \in K$ be arbitrary. Then ${\uparrow}y \subseteq K$, so $(a,b) \in \theta(K) \subseteq \theta({\uparrow}y)$, which means that $y \in \|s_a = s_b\|$.
\end{proof}
Next, we recall a universal algebraic version of the Chinese Remainder Theorem, cf. \cite[Ex.~5.68]{Gratzer} and \cite[Lem.~1.1]{Wol1974}. 
\begin{lemma}\label{lem:chirem}
Suppose $\theta_1,\dots,\theta_n$ are congruences on an algebra $A$ that generate a distributive sublattice of $\Con A$ in which any two congruences commute. Suppose further that $a_1,\dots,a_n \in A$ are such that $(a_i,a_j) \in \theta_i \circ \theta_j$ for all $1 \leq i, j \leq n$. Then there exists $a \in A$ such that $(a,a_i) \in \theta_i$ for each $1 \leq i \leq n$.
\end{lemma}
%
Combining Lemmas~\ref{lem:KStheta}~and~\ref{lem:chirem}, we obtain the following key result.
\begin{proposition}\label{prop:thetacomesfromsheaf}
Let $\theta \colon \KS Y^{\uparrow} \to \Con(A)$, $p \colon E_\theta \to Y^{\uparrow}$ and $F_\theta$ be as in Definition~\ref{def:Ftheta}. Then:
\begin{enumerate}
\item The map $p \colon E_\theta \to Y^{\uparrow}$ is an \'etale bundle of $\mathcal{V}$-algebras.
\item The assignment $a \mapsto s_a$ is an isomorphism from $A$ to the algebra of continuous global sections of $F_\theta$.
\item For every $K \in \KS Y^\uparrow$, the kernel of the restriction map $A \cong F_\theta Y \to F_\theta K$, $a \mapsto s_a|_K$, is equal to $\theta(K)$.
\end{enumerate}
\end{proposition}
\begin{proof}
1. We first note that for $a,b\in A$, the set $\|s_a = s_b\|$ is open in $Y^{\uparrow}$. Indeed, if $y \in \|s_a = s_b\|$, then $(a,b) \in \theta_y = \theta({\uparrow}y)$, so by Lemma~\ref{lem:KStheta}.1, pick $U$ open and $K'$ compact-saturated in $Y^{\uparrow}$ such that $(a,b) \in \theta(K')$ and ${\uparrow}y \subseteq U \subseteq K'$. By Lemma~\ref{lem:KStheta}.2, since $(a,b) \in \theta(K')$, we get $K' \subseteq \|s_a = s_b\|$, so $y \in U \subseteq \|s_a = s_b\|$.

In particular, for each $a \in A$, the map $s_a \colon Y^{\uparrow} \to E$ is continuous: for each $b\in A$ and $U\in\Omega Y^{\uparrow}$, we have $s_a^{-1}(s_b(U)) = \|s_a = s_b\| \cap U$, showing that the inverse image of any set in $\sB$, the generating set for the topology on $E$, is open. Since each $e \in E$ is of the form $e = a/{\theta_y}$ for some $a \in A$ and $y \in Y$, it now also follows that $p$ is \'etale, since $p|_{\im s_a}$ and $s_a$ are mutually inverse continuous maps between $Y^{\uparrow}$ and $\im s_a$. 
Notice also that, for each $n$-ary operation of $\mathcal{V}$-algebras $f$, the partial function $f^E \colon E^n \to E$ is continuous on its domain. Indeed, let $y \in Y$, $\overline{e} \in (E_y)^n$, and $s_a(U) \in \mathcal{B}$ be such that $f^E(\overline{e}) \in s_a(U)$, i.e., $f^E(\overline{e}) = s_a(y)$. For each $i = 1,\dots,n$, pick $b_i \in A$ such that $e_i = [b_i]_{\theta_y}$. Then $(f^{A}(b_1,\dots,b_n), a) \in \theta_y$. By the previous paragraph, the set $V := \|s_{f(b_1,\dots,b_n)} = s_a\|$ is a $Y^{\uparrow}$-open neighborhood of $y$. Now $s_{b_1}(U \cap V) \times \cdots \times s_{b_n}(U \cap V)$ is an open neighborhood of $\overline{e}$ in $E^n$ whose intersection with $\dom(f^E)$ is contained in $(f^E)^{-1}(s_a(U))$, as required.

2. The homomorphism
\begin{align} \label{eq:eta}
\eta \colon A &\to \prod_{y \in Y} A/{\theta_y}, \\
a &\mapsto s_a.\nonumber
\end{align}
is injective. Indeed, by Lemma~\ref{lem:KStheta}.2 applied to $K = Y$, if $\eta(a) = \eta(b)$, then $(a,b) \in \theta(Y)$, but $\theta(Y) = \Delta_A$ since $\theta$ preserves bottom, so $a = b$.

We show that the image of $\eta$ is $\Gamma Y$. We already noted in the proof of the first item that $s_a$ is continuous for every $a \in A$. Conversely, let $s \colon Y \to E$ be a continuous global section of $p$. 
For each $y \in Y$, pick $a_y \in A$ such that $s(y) = [a_y]_{\theta_y}$. Since $\|s = s_{a_y}\|$ is open and $Y^{\uparrow}$ is locally compact, pick $U_y  \in \Omega Y^{\uparrow}$ and $K_y\in {\KS}Y^{\uparrow}$ with $y\in U_y\subseteq K_y\subseteq \|s = s_{a_y}\|$. 
By compactness of $Y$, pick a finite $F \subseteq Y$ such that $(U_y)_{y \in F}$ covers $Y$. 
Note that, for any $y, z \in Y$, we have $K_{y} \cap K_{z} \subseteq \|s = s_{a_y}\| \cap \|s = s_{a_{z}}\| \subseteq \|s_{a_y} = s_{a_{z}}\|$. 
Using Lemma~\ref{lem:KStheta}.2 and the assumption that any two congruences in the image of $\theta$ commute, 
\[(a_y,a_z) \in \theta(K_y \cap K_z) = \theta(K_y) \vee \theta(K_z) = \theta(K_y) \circ \theta(K_z).\]
Since $\theta$ is a homomorphism, the sublattice of $\Con A$ generated by the congruences $(\theta(K_y))_{y  \in F}$ is the image under $\theta$ of the sublattice generated by $(K_y)_{y \in F}$; hence, it is distributive and any congruences in it commute pairwise. By Lemma~\ref{lem:chirem}, pick $a \in A$ such that $(a,a_y) \in \theta(K_y)$ for each $y \in F$. Now, for any $z \in Y$, pick $y \in F$ such that $z \in U_y \subseteq K_y$, and notice that $s(z) = s_{a_y}(z) = s_a(z)$. Thus, $\eta(a) = s$.

3. Note that $s_a|_K = s_b|_K$ if, and only if, $K \subseteq \|s_a = s_b\|$, which, by Lemma~\ref{lem:KStheta}.2, is equivalent to $(a,b) \in K$.
\end{proof}

We are now ready to prove our main theorem, by relativizing the result in Proposition~\ref{prop:thetacomesfromsheaf}.
\begin{theorem}\label{thm:main}
The assignment $F \mapsto \theta_F$ is a bijection between isomorphism classes of soft sheaf representations of $A$ over $Y^{\uparrow}$ and frame homomorphisms from $(\KS Y^{\uparrow})^\op$ to into subframes of $\Con A$ of pairwise commuting congruences.
\end{theorem}

\begin{proof}
Let $\theta \colon (\KS Y^{\uparrow})^\op \to \Con A$ be a frame homomorphism such that any two congruences in the image of $\theta$ commute. By Proposition~\ref{prop:thetacomesfromsheaf}, $F_\theta$ is a sheaf representation of $A$ over $Y^{\uparrow}$ such that $\theta_{F_\theta} = \theta$. It remains to show that (1) $F_\theta$ is soft, (2) $F_\theta$ is up to isomorphism the unique soft sheaf representation of $A$ such that $\theta_{F_\theta} = \theta$.

1. Let $Z$ be any compact saturated set in $Y^{\uparrow}$. We need to prove that the restriction map $\Gamma Y \to \Gamma Z$ is surjective. Note that $Z$, with the subspace topology from $Y^{\uparrow}$, is a stably compact space, with patch topology and order the restrictions of the compact ordered space structure on $Y$ (cf., e.g., \cite[Prop~9.3.4]{Gou2013}). Let $A' := A/{\theta(Z)}$ and let $\theta' \colon \KS Z^{\uparrow} \to \Con(A')$ be defined, for $K \in \KS Z^{\uparrow}$, by $\theta'(K) := \theta(K)/\theta(Z)$. Then $\theta'$ is a frame homomorphism into a subframe of pairwise commuting congruences of $\Con(A')$, since $\KS Z^{\uparrow}$ is isomorphic to the interval $[\emptyset,Z]$ in $\KS Y^{\uparrow}$, and $\Con(A')$ is isomorphic to the interval $[\theta(Z),\nabla_A]$ in $\Con(A)$ (cf., e.g., \cite[Thm. 6.20]{BurSan2000}). We may therefore apply Definition~\ref{def:Ftheta} to the map $\theta' \colon \KS Z^{\uparrow} \to \Con(A')$ to obtain a sheaf $F_{\theta'}$. Notice from the definitions that $F_{\theta'}$ is the restriction of the sheaf $F_{\theta}$ to the subspace $Z$ of $Y$. By Proposition~\ref{prop:thetacomesfromsheaf}.2, applied to $\theta'$, the algebra of global sections of this restricted sheaf is $A' = A/\theta(Z)$. The restriction map $\Gamma Y \to \Gamma Z$ is isomorphic to the quotient map $A \to A'$, which is clearly surjective.

2. To show that $F$ is unique up to isomorphism, let $\tilde{F}$ be any soft sheaf representation of $A$ with $\theta_{\tilde{F}} = \theta$. For any $K \in \KS Y^{\uparrow}$, the algebras of local sections over $K$ for both $F$ and $\tilde{F}$ are isomorphic to $A/{\theta(K)}$, and these isomorphisms are natural in $K$. It follows from Lemma~\ref{lem:sheafiso} that the sheaves $F$ and $\tilde{F}$ are naturally isomorphic.
\end{proof}

Since $(\KS Y^{\uparrow})^\op$ and $\Omega Y^{\downarrow}$ are isomorphic for a compact ordered space $Y$ (see Proposition~\ref{prop:toggle}), we can reformulate Theorem~\ref{thm:main} in terms of $\Omega Y^{\downarrow}$. To this end, given a soft sheaf representation $F$ of $Y$, define 
\begin{align*}
\psi_F \colon \Omega Y^{\downarrow} &\to \Con A \\
U &\mapsto \theta_F(Y \setminus U) = \{(a,b) \in A^2 \ | \ U \cup \|a = b\| = Y\}.
\end{align*}
We obtain the following corollary.

\begin{corollary}\label{cor:main1}
The assignment $F \mapsto \psi_F$ is a bijection between isomorphism classes of soft sheaf representations of $A$ over $Y^{\uparrow}$ and frame homomorphisms $\Omega Y^{\downarrow} \to \Con A$ into subframes of pairwise commuting congruences.
\end{corollary}

We end this section by drawing a further corollary from Corollary~\ref{cor:main1}, which will connect our results to those of \cite{Wol1974,Vag1992}. Note that for algebras in a congruence-permutable variety\footnote{A variety in which any two congruences on an algebra of the variety commute.}, any subframe will do in the corollary above. Now, since any congruence lattice is an algebraic lattice, finite meets always distribute over directed joins. Thus, as soon as the algebra is congruence distributive, it follows that $\Con A$ is a frame. Further, since every congruence is an intersection of completely meet-irreducible congruences by Birkhoff's subdirect decomposition theorem (see, e.g., \cite[Thm.~II.8.6]{BurSan2000}), in particular, every element is an intersection of meet-irreducible elements of  $\Con A$. That is, in a  congruence-permutable and congruence-distributive variety, the congruence lattices are always spatial frames of pairwise commuting congruences. Finally, since congruence lattices are algebraic lattices, the compact elements form a basis. 

In the frame of opens of a stably compact space, since the compact saturated sets of the space are closed under finite intersections, it follows that the compact-open sets are closed under finite intersections as well. Thus, if a congruence lattice is isomorphic to the frame of opens of a stably compact space, then the compact-opens form a sublattice and a basis, that is, it is in fact the frame of opens of a spectral space. Conversely, as soon as the set of compact elements of a spatial frame is closed under finite meets and is join-generating, it is in fact the frame of opens of a spectral space. 

Recall that $A$ is said to have the \emph{Compact Intersection Property} (CIP) provided the intersection of two compact congruences on $A$ is again compact. 
The preceding two paragraphs yield the following general corollary of Corollary~\ref{cor:main1}.

\begin{corollary}\label{cor:main2}
Let $\mathcal V$ be a congruence-permutable and congruence-distributive variety. Then, for any algebra $A$ in $\mathcal V$, the lattice $\Con A$ is isomorphic to $\Omega Y^{\downarrow}$ for some Priestley space $Y$ if, and only if, $A$ has the CIP. In this case, $A$ has a soft sheaf representation over $Y^{\uparrow}$.
\end{corollary}

\section{Direct image sheaves and representations over varying spaces}\label{sec:morphisms}
In this short section, we consider how varying the base space of the sheaf and constructing a direct image sheaf is reflected at the level of frames of pairwise commuting congruences. We will apply the main result of this section, Theorem~\ref{thm:dirimsheaf}, to sheaf representations of MV-algebras in Section~\ref{sec:applications}.

Let $Y_1$ and $Y_2$ be compact ordered spaces, $f\colon Y_1\to Y_2$ a function, and $F_1\colon\Omega Y_1^{\uparrow}\to \sV$ a soft sheaf representation of an algebra $A$.
If $f^{\uparrow}\colon Y_1^{\uparrow}\to Y_2^{\uparrow}$ is continuous, then we obtain a sheaf
\[
F_2=F_1\circ\Omega f^{\uparrow}\colon Y_2^{\uparrow}\to \sV,
\]
known as the direct image sheaf under $f$ obtained from $F_1$. However, it is not clear in general whether $F_2$ is soft even if $F_1$ is.\\

As we have seen in Corollary~\ref{cor:main1}, soft sheaf representations of $A$ over a stably compact space $Y^{\uparrow}$ correspond to frame homomorphisms $\psi\colon \Omega Y^{\downarrow}\to \Con A$ into a frame of pairwise commuting congruences of $A$.
Now, suppose $F_1$ is a soft sheaf representation of $A$ and let $\psi_1 := \psi_{F_1} \colon\Omega Y_1^{\downarrow}\to \Con A$ be the corresponding frame homomorphism. Suppose further that $f^{\downarrow}\colon Y_1^{\downarrow}\to Y_2^{\downarrow}$ is continuous. In this case, we obtain a frame homomorphism into a frame of pairwise commuting congruences of $A$
\[
\psi_2 :=\psi_1\circ\Omega f^{\downarrow}\colon \Omega Y_2^{\downarrow}\to \Con A.
\]
Thus, using Corollary~\ref{cor:main1}, the soft sheaf representation $F_1$ over $Y_1^{\uparrow}$ yields a soft sheaf representation $F_{\psi_2}$ of $A$ over $Y_2^{\uparrow}$. However, it is not clear in general whether $F_{\psi_2}$ is a homomorphic image, and in particular a direct image sheaf, of the sheaf $F_1$.\\

If $f\colon Y_1\to Y_2$ is a morphism of compact ordered spaces, that is, if it is both continuous and order preserving, then we have the following theorem.

\begin{theorem}\label{thm:dirimsheaf}
Let $Y_1$ and $Y_2$ be compact ordered spaces, $f\colon Y_1\to Y_2$ a morphism of compact ordered spaces and $F_1\colon\Omega Y_1^{\uparrow}\to \sV$ a soft sheaf representation of $A$ with corresponding frame homomorphism $\psi_1\colon\Omega Y_1^{\downarrow}\to \Con A$.  Then $F_2=F_1\circ\Omega f^{\uparrow}\colon \Omega Y_2^{\uparrow}\to \sV$ is a soft sheaf representation of $A$ and the corresponding frame homomorphism is $\psi_2=\psi_1\circ\Omega f^{\downarrow}\colon \Omega Y_2^{\downarrow}\to \Con A$.
\end{theorem}

\begin{proof}
Denote by $\theta_1$ and $\theta_2$ the functions $(\KS Y^{\uparrow})^\op \to \Con A$ defined by $\theta_i(K) := \psi_i(Y \setminus K)$. 
By Theorem~\ref{thm:main}, pick a soft sheaf representation $G$ of $A$ over $Y_2^{\uparrow}$ such that $\theta_G = \theta_2$. We prove that $G$ is naturally isomorphic to $F_2$.

Let $U$ be open in $Y_2^{\uparrow}$. 
By definition, $F_2 U =F_1 f^{-1}(U)$. Since $Y_2^{\uparrow}$ is locally compact and $F_1$ is a  soft sheaf representation of $A$ with corresponding frame homomorphism $\theta_1$, Lemma~\ref{lem:sheafiso} gives 
\[
F_2 U = F_1 f^{-1}(U)=\varprojlim\{A/\theta_1(M)\mid M\in\KS Y_1^{\uparrow}\text{ and } M\subseteq f^{-1}(U)\}.
\]
On the other hand, since $\theta_G = \theta_2$,
\begin{align*}
GU& =\varprojlim\{A/\theta_2(K)\mid K\in\KS Y_2^{\uparrow}\text{ and } K\subseteq U\}\\
                         & =\varprojlim\{A/\theta_1(f^{-1}(K))\mid K\in\KS Y_2^{\uparrow}\text{ and } K\subseteq U\}.
\end{align*}
Thus, to show that $G$ and $F_2$ are naturally isomorphic, it suffices to show that the filtering limit systems 
\[
\sS=\{A/\theta_1(f^{-1}(K))\mid K\in\KS Y_2^{\uparrow}\text{ and } K\subseteq U\}
\]
and 
\[
\sT=\{A/\theta_1(M)\mid M\in\KS Y_1^{\uparrow}\text{ and } M\subseteq f^{-1}(U)\}
\]
are equivalent. 
To this end we first note that $\sS\subseteq\sT$. On the other hand, let $M \in \KS Y_1^{\uparrow}$ be such that $M\subseteq f^{-1}(U)$. Then $f[M]\subseteq U$, and thus $K:={\uparrow}f[M]\subseteq U$. Also, since $Y_2$ is Hausdorff, $f[M]$ is compact and thus closed in $Y_2$ and thus $K$ is compact-saturated in $Y_2^{\uparrow}$. By construction we have $M \subseteq f^{-1}(K)$, 
and thus $\sS$ is filtering in $\sT$.
\end{proof}

\section{Applications to distributive-lattice-ordered algebras}\label{sec:applications}
In this section, we apply Theorem~\ref{thm:main} and its corollaries to the specific setting of algebras with a distributive lattice reduct. First, we recall basic facts about distributive lattices and Stone-Priestley duality. We then prove a new, purely duality-theoretic, result on commuting congruences (Lemma~\ref{lem:interpolate}), which may also be of independent interest. We combine this result with Corollary~\ref{cor:main1} to obtain our main theorem about sheaf representations of distributive lattices (Theorem~\ref{thm:DLmain}). We end with illustrating how several sheaf representations of MV-algebras and commutative Gelfand rings available in the literature may be recovered using the general results.

Stone \cite{Sto1937} showed that any distributive lattice\footnote{In this paper we will assume all distributive lattices to be bounded, so we drop the adjective `bounded' for readability. This restriction is not necessary but it is convenient.} $A$ is isomorphic to the lattice of compact-open subsets of a topological space $X$. Moreover, there is up to homeomorphism a unique such \emph{spectral} space, i.e., a stably compact space in which the compact-open sets form a basis for the topology; we call this space $X$ the \emph{Stone spectrum} of $A$. As in the spectral theory of rings, the points of $X$ may be identified with \emph{prime ideals} of $A$, and any element $a \in A$ gives a compact-open set $\hat{a} \subseteq X$ of prime ideals not containing $a$; the assignment $a \mapsto \widehat{a}$ is an isomorphism between the lattice $A$ and the lattice of compact-open sets of $X$. We note that the order of inclusion on the prime ideals is the {\it reverse} of the order of specialization of the Stone spectrum $X$.

Since the Stone spectrum $X$ is in particular a stably compact space, recall from Section~\ref{sec:scs} that $X$ is $X^{\downarrow}$ for a unique compact ordered space $(X,\pi,\leq)$.\footnote{We choose the orientation of the order which fits with the inclusion of prime ideals rather than the order of specialization of the Stone spectrum.}  We call the latter the \emph{Priestley spectrum} of $A$, after Priestley \cite{Pri1970}, who characterized the compact ordered spaces arising in this manner as those which are \emph{totally order-disconnected}: whenever $x, x' \in X$ and $x \nleq x'$, there exists a clopen down-set $K$ of $X$ containing $x'$ and not $x$. The compact-open sets of the Stone spectrum $X^{\downarrow}$ are exactly the clopen down-sets of the Priestley spectrum. By the results cited in Section~\ref{sec:scs} (which post-date Priestley's results), the Stone and Priestley spectra of a distributive lattice are inter-definable. Still, some facts are more easily formulated using the Priestley spectrum, in particular the following result.


\begin{theorem}[Duality between congruences and closed subspaces \cite{Pri1970, Pri1972}]\label{thm:SP}
Let $A$ be a distributive lattice and let $(X,\pi,\leq)$ be the Priestley spectrum of $A$. The assignment
\[ C \mapsto \{(a,b) \in A \times A \ | \ \hat{a} \cap C = \hat{b} \cap C\}\]
is an isomorphism from $(\Cl X)^\op$ to $\Con A$, where $\Cl X$ is the dual frame of closed subsets of $(X,\pi)$ ordered by inclusion.
\end{theorem}

\begin{corollary}\label{cor:SP}
The congruence lattice $\Con A$ of a distributive lattice $A$ is isomorphic to the frame of open sets of the space $(X,\pi)$ underlying the Priestley spectrum of $A$.
\end{corollary}
\begin{proof}
Compose the isomorphism of Theorem~\ref{thm:SP} with the isomorphism between $(\Cl X)^\op$ and $\Omega X$ given by complementation.
\end{proof}

The correspondence between closed sets and congruences can be viewed as a consequence of the \emph{duality} (contravariant equivalence) between the categories of distributive lattices and Priestley spaces. We recall another related result from duality theory, which is not hard to prove directly.
\begin{proposition}[{\cite[Prop.~7]{PP1958}}]\label{prop:pointfree}
Let $X$ and $Y$ be $T_0$ sober spaces. There is a bijection between the set of continuous functions from $X$ to $Y$ and the set of frame homomorphisms from $\Omega Y$ to $\Omega X$, which sends a continuous function $q \colon X \to Y$ to the frame homomorphism $q^{-1} \colon \Omega Y \to \Omega X$.
\end{proposition}
The last insight that we need in order to prove our main theorem about distributive lattices is the following lemma, which, to the best of our knowledge, is brand new.
\begin{lemma}\label{lem:interpolate}
Let $A$ be a distributive lattice and $X$ its Priestley spectrum. Let $\theta_1, \theta_2$ be congruences on $A$ and let $C_1$, $C_2$ be the corresponding closed subsets of $X$, respectively. The following are equivalent:
\begin{enumerate}
\item The congruences $\theta_1$ and $\theta_2$ commute;
\item For any $x_1 \in C_1$, $x_2 \in C_2$, if $\{i, j\} = \{1,2\}$ and $x_i \leq x_j$ then there exists $z \in C_1 \cap C_2$ such that $x_i \leq z \leq x_j$.
\end{enumerate}
\end{lemma}
\begin{proof}
(1) $\Rightarrow$ (2). Let $x_1 \in C_1$, $x_2 \in C_2$, and without loss of generality suppose $i = 1$, $j = 2$. Now suppose that ${\uparrow} x_1 \cap {\downarrow} x_2 \cap C_1 \cap C_2 = \emptyset$; we prove that $x_1 \nleq x_2$. Since $X$ is totally order-disconnected, we have ${\uparrow} x_1 = \bigcap \{X \setminus \hat{a} \ | \ x_1 \not\in \hat{a}\}$ and ${\downarrow} x_2 = \bigcap \{\hat{b} \ | \ x_2 \in \hat{b}\}$. Note that these are intersections of filtered families. Therefore, since $X$ is compact, there exist $a, b \in A$ such that $x_1 \not\in \hat{a}$, $x_2 \in \hat{b}$, and $(X \setminus \hat{a}) \cap \hat{b} \cap C_1 \cap C_2 = \emptyset$. This means that $\hat{b} \cap C_1 \cap C_2 = \hat{a} \cap \hat{b} \cap C_1 \cap C_2$, so the elements $b$ and $a \wedge b$ are identified by the congruence corresponding to $C_1 \cap C_2$, which, by Theorem~\ref{thm:SP}, is $\theta_1 \vee \theta_2$. Since $\theta_1$ and $\theta_2$ commute, we have $\theta_1 \vee \theta_2 = \theta_2 \circ \theta_1$, so pick $c \in A$ such that $b \,{\theta_2}\, c \,{\theta_1}\, (a \wedge b)$. Since $x_2 \in \hat{b} \cap C_2$, we have $x_2 \in \hat{c}\, \cap C_2$ since $b \,{\theta_2}\, c$. On the other hand, $x_1 \in C_1$ and $x_1 \not\in \hat{a}$, so $x_1 \not\in \hat{c}$ since $c \,{\theta_1}\, (a \wedge b)$. Since $\hat{c}$ is a down-set, it follows that $x_1 \nleq x_2$.

(2) $\Rightarrow$ (1). Let $a, b \in A$ be such that $a\,({\theta_1}\,{ \circ}\,{ \theta_2})\,b$. Pick $c \in A$ such that $a\,{\theta_1}\,c\,{\theta_2}\,b$. Consider the following two closed subsets of $X$:
\begin{align*} 
K &:= {\downarrow} \left( (\hat{a} \cap C_2) \cup (\hat{b} \cap C_1) \right),\\
L &:= {\uparrow} \left( (C_2 \setminus \hat{a}) \cup (C_1 \setminus \hat{b}) \right).
\end{align*}

\noindent {\bf Claim.} $K$ and $L$ are disjoint.

\noindent {\it Proof of Claim.} A simple calculation  shows that
\[ K \cap L = \left( {\downarrow} (\hat{a} \cap C_2) \cap {\uparrow} (C_1 \setminus \hat{b}) \right) \cup \left( {\downarrow}(\hat{b} \cap C_1) \cap {\uparrow} (C_2 \setminus \hat{a})\right).\]
Reasoning towards a contradiction, suppose that $x \in K \cap L$, and without loss of generality assume $x \in {\downarrow} (\hat{a} \cap C_2) \cap {\uparrow} (C_1 \setminus \hat{b})$. Pick $x_1 \in C_1 \setminus \hat{b}$ and $x_2 \in \hat{a} \cap C_2$ such that $x_1 \leq x \leq x_2$. By (2), pick $z \in C_1 \cap C_2$ such that $x_1 \leq z \leq x_2$. Since $x_2 \in \hat{a}$ and $\hat{a}$ is a down-set, we have $z \in \hat{a}$. Since $z \in C_1$, and $a \,{\theta_1}\, c$, we have $z \in \hat{c}$. Since $z \in C_2$ and $c {\theta_2} b$, we have $z \in \hat{b}$. However, $x_1 \leq z$ and $x_1 \not\in \hat{b}$, which is a contradiction.\qed

By the claim and the order-normality of Priestley spaces \cite[Lem.~11.21(ii)(b)]{DavPri2002}, there exists $d \in A$ such that $K \subseteq \hat{d}$ and $L \cap \hat{d} = \emptyset$. It now follows from the definitions of $K$ and $L$ that $\hat{a}\, \cap C_2 = \hat{d}\, \cap C_2$ and $\hat{d}\, \cap C_1 = \hat{b}\, \cap C_1$, so that $a\,{\theta_2}\,d\,{\theta_1}\,b$, and $a\,({\theta_2}\,{\circ}\,{\theta_1})\,b$, as required.
\end{proof}

We now come to the main definition of this section.
\begin{definition}
Let $(X,\pi,\leq_X)$ be a Priestley space and $(Y,\tau,\leq_Y)$ a compact ordered space. We say a continuous function $q \colon X \to Y^{\downarrow}$ is an \emph{interpolating decomposition} of $X$ over $Y$ if, for all $x_1, x_2 \in X$, if $x_1 \leq_X x_2$, then there exists $z \in X$ such that $q(x_1)\leq_Y q(z)$, $q(x_2)\leq_Y q(z) $ and $x_1 \leq_X z \leq_X x_2$. 
%
\end{definition}
If $X$ is the Priestley spectrum of a distributive lattice $A$ and $q \colon X \to Y^{\downarrow}$ is a continuous function, denote by $\psi_q \colon \Omega Y^{\downarrow} \to \Con A$ the function obtained by composing the frame homomorphism $q^{-1} \colon \Omega Y^{\downarrow} \to \Omega X$ with the frame isomorphism $\Omega X \cong \Con A$ given in Corollary~\ref{cor:SP}.
\begin{proposition}\label{prop:intdec}
The following are equivalent:
\begin{enumerate}
\item The function $q$ is an interpolating decomposition;
\item Any two congruences in the image of $\psi_q$ commute.
\end{enumerate}
\end{proposition}
\begin{proof}
(1) $\Rightarrow$ (2). Let $U_1, U_2 \in \Omega Y^{\downarrow}$. To show that $\psi_q(U_1)$ and $\psi_q(U_2)$ commute, it suffices to prove that the closed subsets $C_i := q^{-1}(Y \setminus U_i)$ ($i = 1,2$) satisfy property (2) in Lemma~\ref{lem:interpolate}. Let $x_i \in C_i$ and suppose without loss of generality that $x_1 \leq_X x_2$. By assumption, pick $z$ such that $x_1 \leq_X z \leq_X x_2$, $q(x_1)\leq_Y q(z)$ and $q(x_2)\leq_Y q(z) $ for $i = 1,2$. Since $U_i$ for $i = 1,2$ are open in $Y^{\downarrow}$, they are down-sets in the order on $Y$. It follows that $Y \setminus U_i$ is an up-set and since $q(x_i)\in Y \setminus U_i$  for $i = 1,2$, it follows that $q(z)\in (Y \setminus U_1)\cap (Y \setminus U_2)$ so that $z \in C_1 \cap C_2$, as required.

(2) $\Rightarrow$ (1). Let $x_1, x_2 \in X$ be such that $x_1 \leq_X x_2$. Write $y_i := q(x_i)$ and $C_i := q^{-1}({\uparrow}y_i)$ for $i = 1,2$. By continuity of $q$, $C_1$ and $C_2$ are closed, and clearly $x_i \in C_i$ for $i = 1, 2$. Moreover, note that, by definition of $\psi_q$, $C_i$ is the closed subset corresponding to the congruence $\psi_q(Y \setminus {\uparrow}y_i)$ under the isomorphism of Theorem~\ref{thm:SP}. The congruences $\psi_q(Y \setminus {\uparrow}y_1)$ and $\psi_q(Y \setminus {\uparrow}y_2)$ commute by assumption, so by Lemma~\ref{lem:interpolate}, there exists $z \in C_1 \cap C_2$ such that $x_1 \leq_X z \leq_X x_2$. The fact that $z\in C_i$ is equivalent to $q(x_i)\leq_Y q(z)$ for $i = 1, 2$, as required.
\end{proof}

We are now ready for the main theorem of this section. Let $A$ be a distributive lattice with Priestley spectrum $X$. If $F$ is a sheaf representation of $A$ over a stably compact space $Y^{\uparrow}$, recall that at the end of 
Section\,\ref{sec:main} we defined the function $\psi_F \colon \Omega Y^{\downarrow} \to \Con A$. By Corollary~\ref{cor:main1}, $\psi_F$ is a frame homomorphism. Denote by $\chi_F$ the frame homomorphism $\Omega Y^{\downarrow} \to \Omega X$ obtained by composing $\psi_F$ with the isomorphism $\Con A \cong \Omega X$ from Corollary~\ref{cor:SP}. By Proposition~\ref{prop:pointfree}, let $q_F \colon X \to Y^{\downarrow}$ be the unique continuous function such that $\chi_F = (q_F)^{-1}$.

\begin{theorem}\label{thm:DLmain}
Let $A$ be a distributive lattice with dual Priestley space $X$. The assignment $F \mapsto q_F$ is a bijection between isomorphism classes of soft sheaf representations of $A$ over $Y^{\uparrow}$ and interpolating decompositions of $X$ over $Y^{\downarrow}$.
\end{theorem}
\begin{proof}
Note that, by Proposition~\ref{prop:intdec}, the image of the composition of the bijections of Corollary~\ref{cor:main1} and Proposition~\ref{prop:pointfree} consists exactly of interpolating decompositions.
\end{proof}

\begin{theorem}\label{thm:dirimsheafdual}
Let $A$ be a distributive lattice with Priestley spectrum $X$, and let $q\colon X\to Y^{\downarrow}$ be an interpolating decomposition of $X$. Any map $f\colon Y\to Z$ between compact ordered spaces which is continuous with respect to the down-topologies on $Y$ and $Z$ yields an interpolating decomposition of $X$ over $Z^{\downarrow}$ and thus a soft sheaf representation of $A$ over $Z^{\uparrow}$. If $f$ is also continuous with respect to the up-topologies, then the soft sheaf representation of $A$ over $Z^{\uparrow}$ is the direct image sheaf given by $f$ of the soft sheaf representation of $A$ over $Y^{\uparrow}$ corresponding to $q$.
\end{theorem}

\begin{proof}
This is a direct consequence of Theorem~\ref{thm:dirimsheaf}, the comments preceding it and Theorem~\ref{thm:DLmain}.
\end{proof}

We end the paper by connecting our results to three concrete instances in the literature, namely MV-algebras, Gelfand rings, and distributive lattices themselves.\\

{\bf MV-algebras.} 
For readers familiar with MV-algebras and their sheaf representations, we show how the results of this paper apply to that setting. (For definitions and background on MV-algebras, cf. \cite{CDM,Mun2011}.) The variety of MV-algebras is congruence distributive since MV-algebras have a (distributive) lattice reduct and the variety of lattices is congruence distributive. Furthermore, the variety of MV-algebras is congruence-permutable (a fact that is sometimes referred to as the \emph{Chinese remainder theorem for MV-algebras}). Finally, the variety of MV-algebras satisfies CIP. In fact, the map
\begin{align}\label{eq:homtokcon}
 \lambda \colon A &\to \mathrm{KCon} A\\
 a &\mapsto \theta(0,a)\nonumber
\end{align}
is a lattice homomorphism (see e.g. \cite[Proposition~4.3]{GGM2014}) onto $\mathrm{KCon} A$, the lattice of compact congruences of $A$. That is, all finitely generated congruences are principal and $\theta(0,a)\cap\theta(0,b)=\theta(0,a\wedge b)$ so that the intersection of two compact congruences is again compact. Thus it follows from Corollary~\ref{cor:main2} that $\Con A$ is isomorphic to $\Omega Y^{\downarrow}$ where $Y$ is the Priestley spectrum of the lattice $\mathrm{KCon} A$, and $A$ has a soft sheaf representation over $Y^{\uparrow}$. 

The Priestley dual of the homomorphism $\lambda$ defined in (\ref{eq:homtokcon}) is an embedding of $Y$ into $X$, the Priestley spectrum of the distributive lattice reduct of $A$. Indeed, an alternative description of $Y$, more common in the literature on MV-algebras, is that it is the space of prime MV-ideals of $A$, which is a closed subspace of the Priestley space $X$ of prime lattice ideals of $A$. 
The compact-open sets of $Y^{\downarrow}$ are those of the form $\hat{a} = \{\mathfrak{p} \in Y \ | \ a \not\in \mathfrak{p}\}$, for $a \in A$. This space $Y^{\downarrow}$ is what is known in the literature on MV-algebras as the \emph{MV spectrum} of $A$ endowed with the Zariski topology. In this presentation, the isomorphism between $\Omega Y^{\downarrow}$ and $\Con A$ can be given explicitly by
\begin{align*}
\Omega Y^{\downarrow} &\to \Con A\\
U &\mapsto \bigcap_{\mathfrak{p}\not\in U} \theta_{\mathfrak{p}},
\end{align*}
where $\theta_{\mathfrak{p}}$ is the kernel of $A \to A/{\mathfrak{p}}$, which identifies $a, b \in A$ iff $(a \ominus b) \oplus (b \ominus a) \in \mathfrak{p}$.

Note that the sheaf representation obtained as explained above from Corollary~\ref{cor:main2} is a sheaf representation of $A$ over the MV spectrum endowed with the \emph{co-Zariski} topology. This is in fact the sheaf representation for MV-algebras presented in \cite{DP2010} by Dubuc and Poveda.

In \cite{GGM2014}, part of the results presented here were first developed to analyze sheaf representations of MV-algebra. Indeed, there, an interpolating map from $X$ to $Y^{\downarrow}$ was exhibited and the sheaf representation discussed above was seen to be definable from this map. Theorem~\ref{thm:DLmain} tells us that it is precisely \emph{soft} sheaf representations that are obtainable in this way. Given an MV-algebra $A$, the interpolating decomposition $k\colon X\to Y^{\downarrow}$ given in \cite{GGM2014} may be described as follows
\begin{align*}
k \colon X &\to  Y^{\downarrow}\\
\mathfrak{q} &\mapsto  \{a\in A \ | \ \forall c \in \mathfrak{q} \quad a\oplus c\in\mathfrak{q}\}.
\end{align*}
 
The maximal MV-ideals of an MV-algebra $A$, that is, the maximal points of the MV spectrum $Y$ of $A$ form a compact Hausdorff space $Z$ when endowed with the subspace topology from $Y^{\downarrow}$. Furthermore, the order on $Y$ is a root system so that each point $y\in Y$ is below a unique maximal point $m(y)\in Z$. In fact the map $m\colon Y\to Z$ is continuous as a map of stably compact spaces (where $Z$ carries the trivial order). As a consequence of Theorem~\ref{thm:dirimsheafdual}, we obtain a soft sheaf representation of $A$ over $Z$ and this sheaf representation is the direct image sheaf of the sheaf representation of $A$ over $Y^{\uparrow}$ under the map $m$.\\

{\bf Gelfand rings.}
%
A \emph{Gelfand ring} is a ring such that every prime ideal is contained in a unique maximal ideal; equivalently, the frame of radical ideals of the ring is normal \cite[Prop.~V.3.7]{Joh1982}. 
Banaschewski and Vermeulen \cite{BanVer2011}, building on results in \cite{Bko70,Mulvey}, characterize commutative Gelfand rings in terms of their sheaf representations. Their results imply that commutative Gelfand rings are exactly those rings which admit a \emph{locally soft} sheaf representation of local rings over a compact Hausdorff space.
Due to the fact that the base space is normal, local softness as defined in \cite{BanVer2011} is in fact equivalent to softness \cite[Prop.~2.6.2]{Bko70}, \cite[Thm. II.3.7.2]{God1958}. 

The soft sheaf representation of commutative Gelfand rings in \cite{BanVer2011} can be seen as an application of our Corollary~\ref{cor:main1}, as follows. For any commutative Gelfand ring $A$, the frame $\mathrm{Reg}(\mathrm{RId} A)$ of regular radical ideals of $A$ is a subframe of commuting congruences of the congruence lattice of $A$, where we have identified, as usual, the ideals of $A$ with the congruences on $A$. Moreover, the frame $\mathrm{Reg}(\mathrm{RId} A)$ is isomorphic to the open set lattice of the space $\mathrm{max}\,A$ of maximal ideals of $A$ with the Zariski topology, which is a compact Hausdorff space, and is therefore equal to its co-compact dual. Let $\theta \colon \Omega(\mathrm{max}\,A) \into \mathrm{Con}\,A$ be the injective frame homomorphism which sends an open set of maximal ideals to the congruence for the corresponding regular radical ideal. Then Corollary~\ref{cor:main1}  yields a soft sheaf representation of $A$ over $\mathrm{max}\,A$, corresponding to $\theta$, which is exactly the sheaf considered in \cite{BanVer2011}.\\


{\bf Distributive lattices and beyond.} 
The representation theorem for Boolean algebras provided by Stone's duality may be seen as a sheaf representation: Every Boolean algebra $B$ is isomorphic to the algebra of global sections of the sheaf $F\colon\Omega X\to \mathcal{BA}$ where $X$ is the dual space of $B$ and $F(U) := [U,2]$ is the set of all continuous functions from $U$ into the two-element discrete space. This is a soft sheaf representation whose stalks are all isomorphic to the two-element lattice; as a section over $X$ each element $a\,{\in}\,B$ is represented by the characteristic function of the corresponding clopen subset $\widehat{a}\subseteq X$. In this sense sheaf representations may be viewed as a generalization of Stone's representation theorem for Boolean algebras. 

By contrast, the representation theorem for distributive lattices provided by Stone's duality does not correspond to a sheaf representation for these lattices. The identity map from $X$ to $X^{\downarrow}$ is order preserving and thus in particular interpolating. However, the stalks of the corresponding sheaf are the lattices dual to the sets ${\uparrow}x$ for $x\in X$. These are all isomorphic to the two-element lattice if and only if the order on $X$ is trivial, which happens by Nachbin's Theorem \cite{Nachbin47} if and only if the lattice is in fact Boolean. Relative to sheaves, the representation theorem for distributive lattices provided by Stone's duality is more naturally seen through the perspective of Priestley duality: Let $A$ be a distributive lattice and $B$ its Booleanization. Then the topological space reduct of the Priestley space $X$ of $A$ is the Stone space of $B$, and the order on $X$ identifies $A$ as those global sections of the sheaf for $B$ which are not only continuous but also order preserving. Once the stalks have more than two elements, order preserving sections are not the right concept, but rather what was identified by Jipsen as so-called \emph{ac-labellings} \cite{Jip2009}. An investigation of the ensuing notion of so-called Priestley products in the spirit of this paper with applications to GBL-algebras is on-going work by the first author with Peter Jipsen and Anna Carla Russo.

\subsection*{Acknowledgements} The authors thank the anonymous referee for a thorough reading of the paper and many helpful comments and suggestions, which improved the paper. We want to acknowledge in particular that we adopted an alternative proof strategy for Theorem~\ref{thm:main} that was suggested to us by the referee.

\bibliographystyle{amsplain}
\bibliography{GvG2016.bib}

\end{document}